\documentclass[11pt]{amsart}

\usepackage{amsaddr}
\usepackage{amscd,amssymb}
\usepackage{enumitem}
\usepackage{graphics,graphicx}
\setlist[itemize]{leftmargin=*}
\setlist[enumerate]{leftmargin=*}
\setlist[description]{leftmargin=*}

\title[Hom-associative magmas and Hom-associative Magma Algebras]{Hom-associative magmas with applications to Hom-associative
magma algebras}

\newtheorem{thm}{Theorem}
\newtheorem{prop}[thm]{Proposition}

\theoremstyle{definition}
\newtheorem{defi}[thm]{Definition}
\newtheorem{exa}[thm]{Example}

\newcommand{\I}{{\rm id}}

\newcommand{\Pfun}{{\rm Pfun}}

\newcommand{\Fun}{{\rm Fun}}

\newcommand{\M}{{\rm M}}
\newcommand{\PM}{{\rm PM}}
\newcommand{\WPM}{{\rm WPM}}

\begin{document}

\author{Patrik Lundström}
\address{University West,
Department of Engineering Science, 
SE-46186 Trollh\"{a}ttan, Sweden}

\email{patrik.lundstrom@hv.se}



\begin{abstract}
Let $X$ be a magma, that is a set equipped with a 
binary operation, and consider a function
$\alpha : X \to X$. We say that $X$ is 
Hom-associative if, for all $x,y,z \in X$, the equality 
$\alpha(x)(yz) = (xy) \alpha(z)$ 
holds. For every isomorphism class of magmas of order two, we 
determine all functions $\alpha$ making $X$ Hom-associative.
Furthermore, we find all such $\alpha$ that are 
endomorphisms of $X$. We also consider versions
of these results where the binary operation on $X$ as well as the 
function $\alpha$ are partially defined. We use our findings 
to construct many examples of two-dimensional 
Hom-associative and multiplicative magma algebras. 
\end{abstract}

\maketitle

\section{Introduction}

In the last decades there has risen an intense interest 
in various \emph{Hom} versions of algebraical objects. 
The defining axioms of these objects  
are miscellaneous endomorphism deformations of its standard axioms.
The first example of this seems to be
\cite{hartwig2006} where Hartwig, Larsson and Silvestrov define
\emph{Hom–Lie algebras}. For such objects, the usual Jacobi identity 
is replaced by the so called Hom–Jacobi identity:
\[
[\alpha(x),[y,z]]+ [\alpha(y),[z,x]] + [\alpha(z),[x,y]] = 0
\]
where $\alpha$ is an endomorphism of the Lie algebra.
Another instance of this is 
\cite{makhlouf2008} where Makhlouf and Silvestrov introduce
\emph{Hom-algebras}, where the usual 
associativity is replaced by so called Hom-associativity: 
\begin{equation}\label{eq:hom-associativity}
\alpha(x)(yz) = (xy) \alpha(z)
\end{equation}
where $\alpha$ now is an algebra 
endomorphism. Similarly, Hom-coalgebras, Hom-bialgebras and 
Hom–Hopf algebras have been proposed, see 
\cite{makhlouf2009,makhlouf2010,yau2010}.
In \cite{laurent2018} Laurent-Gengoux, Makhlouf and 
Teles define a \emph{Hom-group} as a nonempty set
equipped with a binary operation satisfying 
(\ref{eq:hom-associativity}), multiplicativity of $\alpha$:
\begin{equation}\label{eq:multiplicative}
\alpha(xy) = \alpha(x) \alpha(y)
\end{equation}
and having a distinguished member 1 satisfying
the unital identity: 
\begin{equation}
1x = x1 = \alpha(x)
\end{equation}
as well as some 
Hom versions of invertibility axioms for $X$ 
(see \cite[Def. 0.1]{laurent2018}).

An impetus for studying Hom versions of classical 
mathematical objects is that 
it potentially could give us a language to describe families of 
involved mathematical structures
using well studied less complicated structures, 
looking at them through a Hom lens.
There are many such instances. 
Indeed, in \cite{goze2017} Goze and Remm show that \emph{all}
three-dimensional algebras are Hom-associative Lie algebras.
Another example is \cite[Ex. 2.13-14]{hassanzadeh2020} 
where Hassanzadeh describes several non-associative structures
as Hom-groups. For other relevant results on various types of
Hom-associative structures, see \cite{basdouri2021} and
\cite{jiang2020} and the references therein.

In this article, we apply this philosophy to the 
context of \emph{magmas}, that is 
sets equipped with a binary operation (see \cite[p.~1]{bourbaki1989}). 
Classically, these objects have 
been categorized into many different types of families, such as 
groups, semigroups, Brandt groupoids, quasigroups, 
multigroups, hypergroups, loops etc. (see e.g. \cite{bruck1958}). 
We wish to add a Hom perspective to this classification.
More specifically, for a given magma $X$, we would like  
to answer the following questions:
\begin{itemize}[itemsep=3mm,topsep=3mm]
\item \emph{For what functions  $\alpha: X \to X$
is $X$ Hom-associative in the sense of (\ref{eq:hom-associativity})?}

\item \emph{Which of these functions are magma endomorphisms
in the sense of (\ref{eq:multiplicative})?}
\end{itemize}

To fully answer both of these questions for all magmas
is probably a difficult task. 
So a first step would be to consider some special classes of magmas.
In this article, we completely answer these questions
for magmas of order two
(see Theorem \ref{thm:main}).
Note that finding such Hom structures on magmas is important not only
from the magma perspective, but also from the point of 
view of algebras over a field $K$. Namely, given a
function $\alpha : X \to X$, then it induces 
a natural Hom-algebra structure on the magma algebra $K[X]$ of 
$X$ over $K$, 
and Hom properties of $\alpha$ reflects upon 
algebra properties of $K[X]$ (see Theorem \ref{thm:secondmain}).

Here is a detailed outline of the article.

In Section \ref{sec:homassociativemagmas}, we first state our
conventions on sets, relations, (partial) functions and
(partial) equality of (partial) functions. Then we define various 
concepts of (partial) magmas such as weak/partial homomorphisms,
and (partially) Hom-associative magmas. Thereafter, we
consider magmas of order two. We first find all non-isomorphic
multiplication tables of partial magmas of order two.
Then we give a complete characterisation of all (weak)
partial endomorphisms of these structures as well as 
all (partial) Hom-associative structures defined on them
(see Theorem \ref{thm:main}). 

In Section \ref{sec:magmaalgebras}, we
first recall some classical definitions of multiplicative and 
Hom-associative Hom-algebras (see Definition \ref{def:homproperties}).
After that, we introduce 
partial versions of these concepts 
(see Definition \ref{def:partilconcepts}).
Then we show how various Hom properties of a magma $X$ 
reflect upon properties of the corresponding magma algebra $K[X]$
(see Theorem \ref{thm:secondmain}). At the end of this section,
we exemplify our main results for some instances of 
two-dimensional magma algebras
(see Example \ref{ex:someinstances}).

\section{Hom-associative magmas}\label{sec:homassociativemagmas}

\subsection{Relations and functions}

Let $X$ and $Y$ be sets. Suppose that 
$f$ is a relation from $X$ to $Y$.
By this we mean that $f$ is a subset of $X \times Y$ 
and we denote this by $f : X \to Y$.
The {\it inverse relation of $f$}, denoted by $f^{-1} : Y \to X$, 
is the set $\{ (y,x) \in Y \times X \mid (x,y) \in f \}$.
Given $x \in X$ we put 
$f(x) := \{ y \in Y \mid (x,y) \in f \}$
and we say that $f(x)$ is \emph{defined} when
$f(x) \neq \emptyset$.
The \emph{range} and \emph{domain} of $f$ are defined
to be the sets $R_f := \cup_{x \in X} f(x)$ and
$D_f := \cup_{y \in Y} f^{-1}(y)$ respectively.
We say that $f$ is a \emph{partial function} if for all
$x \in X$ the set $f(x)$ has at most one element.
In that case, if $f(x)$ is defined and $f(x) = \{ y \}$,
then we will often, as customary, write $f(x) = y$.
If $f$ is a partial function with $D_f = X$, then 
$f$ is called a \emph{function}.
If $g : Y \to Z$ is another relation, then the 
{\it composition of $g$ and $f$}, denoted by $g \circ f : X \to Z$,
is the set $\{ (x,z) \in X \times Z \mid \exists y \in Y \ 
(x,y) \in f \ {\rm and} \ (y,z) \in g \}$.
If $h : X' \to Y'$ is yet another relation, then 
$f \times h : X \times X' \to Y \times Y'$
is the relation $\{ ( (x,x') , (y,y') ) \mid (x,y) \in f \ 
{\rm and} \ (x',y') \in h \}$.
The identity relation $\I_X : X \to X$ is the set 
$\{ (x,x) \mid x \in X \}$;
often we will skip the subscript and just write $\I$. 
We let $\Fun(X,Y)$ ($\Pfun(X,Y)$) denote the set of 
(partial) functions from $X$ to $Y$.
Suppose that $f,g \in \Pfun(X,Y)$. We say that $f$ and $g$ are 
\emph{partially equal}, denoted by $f \approx g$, 
if for all $x \in X$ such that $f(x)$
and $g(x)$ are defined, then $f(x) = g(x)$. 
Note that the relation $\approx$ is
reflexive and symmetric but not necessarily transitive.
Clearly, the restriction of $\approx$ to $\Fun(X,Y)$ 
coincides with the ordinary equality of functions.

\subsection{Magmas}

Let $(X,\nabla)$ be a \emph{partial magma}. 
By this we mean that $X$ is a set and
$\nabla \in \Pfun(X \times X , X)$.
Note that if $\nabla \in \Fun(X \times X,X)$,
then $(X,\nabla)$ is a magma.
Let $(X',\nabla')$ be another partial magma and 
suppose that $\alpha \in \Pfun(X,X')$. 

\begin{defi}
With the above notations, we say that $\alpha$ is a:
\begin{itemize}[itemsep=1mm,topsep=1mm]

\item \emph{weak partial homomorphism of partial magmas} if
$\alpha \circ \nabla \approx \nabla' \circ (\alpha \times \alpha)$
as partial functions.
In that case, if $\alpha \in \Fun(X,X')$, then $\alpha$ is
called a \emph{weak homomorphism of partial magmas};

\item \emph{partial homomorphism of partial magmas} if
$\alpha \circ \nabla = \nabla' \circ (\alpha \times \alpha)$
as partial functions.
In that case, if $\alpha \in \Fun(X,X')$, then $\alpha$ is
called a \emph{homomorphism of partial magmas};

\item \emph{homomorphism of magmas} if 
$\alpha$ is a homomorphism of partial magmas and
$(X,\nabla)$ and $(X',\nabla')$ are indeed magmas.

\end{itemize}
\end{defi}

We let $\M$ ($\PM$) denote the category having (partial) magmas 
as objects and (partial) homomorphisms of (partial) magmas as morphisms. 
We let 
$\WPM$ denote the category having partial magmas
as objects and weak 
partial homomorphisms of partial magmas
as morphisms. Clearly, $\M$ is a subcategory of $\PM$
which, in turn, is a subcategory of $\WPM$.
The next result will not be used in the sequel in full 
generality. Nevertheless, we record it for its own interest.

\begin{prop}\label{prop:Miso}
Let $(X,\nabla)$ and $(X',\nabla')$ be partial magmas and
suppose that $\alpha : (X,\nabla) \to (X',\nabla')$ 
is a morphism in $\WPM$.
\begin{itemize}[widest=(a),topsep=1mm,itemsep=1mm]

\item[{\rm (a)}] The map $\alpha$ is an isomorphism in $\M$ if and
only if $\alpha$ is bijective and for all $x,y \in X$
the equality
$\alpha(\nabla(x,y)) = \nabla'(\alpha(x),\alpha(y))$
holds.

\item[{\rm (b)}] The map $\alpha$ is an isomorphism in $\PM$ if and
only if $\alpha$ is bijective and for all $x,y \in X$
$\alpha(\nabla(x,y))$ is defined $\Leftrightarrow$
$\nabla'(\alpha(x),\alpha(y))$ defined, 
and in that case the equality 
$\alpha( \nabla(x,y) ) = \nabla'(\alpha(x),\alpha(y))$
holds.
 
\item[{\rm (c)}] The map $\alpha$ is an isomorphism in $\WPM$ 
if and only if $\alpha|_{D_\alpha}$ is 
bijective and for all $x,y \in X$
$\alpha(\nabla(x,y))$ is defined $\Leftrightarrow$
$\nabla'(\alpha(x),\alpha(y))$ defined, 
and in that case the equality 
$\alpha( \nabla(x,y) ) = \nabla'(\alpha(x),\alpha(y))$
holds.
 
\end{itemize} 
\end{prop}

\begin{proof}
(a) Follows from (b). Now we prove (b).
First we show the ``only if'' statement.
Suppose that $\alpha : (X,\nabla) \to (X',\nabla')$ 
is an isomorphism in $\PM$.
Then there is a morphism $\beta : (X',\nabla') \to (X,\nabla)$
in $\PM$ such that $\beta \circ \alpha = \I_X$ and 
$\alpha \circ \beta = \I_{X'}$. Take $x,y \in X$.
If $\alpha(\nabla(x,y))$ is defined, then, since 
$\alpha$ is morphism in $\PM$ it follows that  
$\nabla'(\alpha(x),\alpha(y))$
is also defined. If $\nabla'(\alpha(x),\alpha(y))$
is defined, then $\beta(\nabla'(\alpha(x),\alpha(y)))$
is defined, which, since $\beta$ is a morphism in $\PM$,
implies that $\nabla( \beta(\alpha(x)),\beta(\alpha(y)) ) =
\nabla(x,y)$ is defined. Thus $\alpha(\nabla(x,y))$ is defined.
Now we show the ``if'' statement. Suppose that
$\alpha$ is bijective and for all $x,y \in X$
$\alpha(\nabla(x,y))$ is defined $\Leftrightarrow$
$\nabla'(\alpha(x),\alpha(y))$ defined, 
and in that case the equality 
$\alpha( \nabla(x,y) ) = \nabla'(\alpha(x),\alpha(y))$
holds. Put $\beta = \alpha^{-1}$ and take $x',y' \in X'$
such that $\beta(\nabla'(x',y'))$ is 
defined. Take $x,y \in X$ with $\alpha(x)=x'$ and $\alpha(y)=y'$.
Then $\beta(\nabla'(\alpha(x),\alpha(y)))$ is defined.
From the assumptions it follows that
$\nabla(x,y) = \beta(\alpha(\nabla(x,y)))$ 
also is defined and that
$\beta(\nabla'(x',y')) = \nabla( \beta(x'),\beta(y') )$.
Thus, $\beta$ is a morphism in $\PM$.
Clearly, $\beta \circ \alpha = \I_X$ and 
$\alpha \circ \beta = \I_{X'}$ so that $\alpha$
is an isomorphism in $\PM$.
The statement in (c) follows from (b) by restriction.
\end{proof}

\begin{defi}
Suppose that $(X,\nabla)$ is a partial magma and
$\alpha \in \Pfun(X,X)$. We say that the triple 
$(X,\nabla,\alpha)$ is: 
\begin{itemize}[itemsep=1mm,topsep=1mm]

\item \emph{partially Hom-associative} if 
$\nabla \circ ( \alpha \times \nabla ) \approx
\nabla \circ ( \nabla \times \alpha)$ as partial functions;

\item \emph{Hom-associative} if 
$\nabla \circ ( \alpha \times \nabla ) =
\nabla \circ ( \nabla \times \alpha)$ as partial functions;

\item \emph{partially associative} if $(X,\nabla,\I)$ 
is partially Hom-associative;

\item \emph{associative} if $(X,\nabla,\I)$ is Hom-associative.

\end{itemize}
\end{defi}

\subsection{Magmas of order two}

For the rest of this section, $(X,\nabla)$ denotes a partial magma
with $X = \{ 1,2 \}$. We will 
write $\nabla(x,y) = 3$ when $\nabla(x,y)$ is 
not defined. A multiplication table defined by
$\nabla(1,1) = a$, $\nabla(1,2) = b$, 
$\nabla(2,1) = c$ and 
$\nabla(2,2) = d$
will be written in short hand as $abcd$. So, for instance,
the multiplication table $2 1 3 1$ is to be interpreted as
$\nabla(1,1) = 2$, $\nabla(1,2) = 1$, 
$\nabla(2,1)$ is undefined and $\nabla(2,2) = 1$
Clearly, there are $3^4 = 81$ different such multiplication tables.
Put $\overline{X} = X \cup \{ 3 \}$ and define the function 
$t : \overline{X} \to \overline{X}$ by 
$t(1)=2$, $t(2)=1$ and $t(3) = 3$.
The next result is probably folklore.
Nonetheless, for the convenience of the reader, we include
it as well as a proof of it here.

\begin{prop}\label{prop:tiso}
Suppose that $a,b,c,d,e,f,g,h \in \overline{X}$. 
The multiplication tables $abcd$ and $efgh$ yield 
isomorphic partial magmas  
if and only if $a=e$, $b=f$, $c=g$ and $d=h$, or
$t(a)=h$, $t(b)=g$, $t(c)=f$ and $t(d)=e$.
\end{prop}

\begin{proof}
Let $\nabla$ and $\nabla'$ denote the partial maps
$X \times X \to X$ defined by the multiplication tables
$abcd$ and $efgh$ respectively.
First we show the ''if'' statement. We consider two cases.
Case 1: $a=e$, $b=f$, $c=g$ and $d=h$. 
By Proposition \ref{prop:Miso}, $\I$
is an isomorphism of partial magmas $(X,\nabla) \to (X,\nabla')$.
Case 2: $t(a)=h$, $t(b)=g$, $t(c)=f$ and $t(d)=e$.
Then: 
\[
\begin{array}{ccccccccl}
t(\nabla(1,1)) &=& t(a) &=& h &=& \nabla'(2,2) &=& 
\nabla'(t(1),t(1)) \\
t(\nabla(1,2)) &=& t(b) &=& g &=& \nabla'(2,1) &=& 
\nabla'(t(1),t(2)) \\
t(\nabla(2,1)) &=& t(c) &=& f &=& \nabla'(1,2) &=& 
\nabla'(t(2),t(1)) \\
t(\nabla(2,2)) &=& t(d) &=& e &=& \nabla'(1,1) &=& 
\nabla'(t(2),t(2)). 
\end{array}
\]
Therefore, by Proposition \ref{prop:Miso} again, 
$t$ is an isomorphism of partial magmas $(X,\nabla) \to (X,\nabla')$.
Now we show the ''only if'' statement. Suppose that 
$F : (X,\nabla) \to (X,\nabla')$ is an isomorphism of
partial magmas. By Proposition~\ref{prop:Miso},
$F$ is a bijection $X \to X$ so that $F = \I$ or $F = t$. 
Case 1: $F = \I$. Then:
\[
\begin{array}{ccccccc}
F( \nabla(1,1) ) & = & \nabla'(F(1),F(1)) &\Rightarrow& a & = & e \\
F( \nabla(1,2) ) & = & \nabla'(F(1),F(2)) &\Rightarrow& b & = & f \\
F( \nabla(2,1) ) & = & \nabla'(F(2),F(1)) &\Rightarrow& c & = & g \\
F( \nabla(2,2) ) & = & \nabla'(F(2),F(2)) &\Rightarrow& d & = & h. \\
\end{array}
\]
Case 2: $F = t$. Then:
\[
\begin{array}{ccccccc}
F( \nabla(1,1) ) &=& \nabla'(F(1),F(1)) &\Rightarrow& t(a) &=& h \\
F( \nabla(1,2) ) &=& \nabla'(F(1),F(2)) &\Rightarrow& t(b) &=& g \\
F( \nabla(2,1) ) &=& \nabla'(F(2),F(1)) &\Rightarrow& t(c) &=& f \\
F( \nabla(2,2) ) &=& \nabla'(F(2),F(2)) &\Rightarrow& t(d) &=& e. \\
\end{array}
\]
\end{proof}

\begin{prop}\label{prop:isomorphismclasses}
The 81 multiplication tables of partial magma structures
defined on $X$ is partitioned into the following 45 isomorphism classes:
\begin{enumerate}[widest=(4)]

\item[(1)] 
$3333$ \
(2) $1333 \cong 3332$ \
(3) $2333 \cong 3331$ \
(4) $3133 \cong 3323$ 

\item[(5)] $3233 \cong 3313$ \
(6) $1133 \cong 3322$ \
(7) $1233 \cong 3312$ \
(8) $2133 \cong 3321$ 

\item[(9)] $2233 \cong 3311$ \
(10) $1313 \cong 3232$ \
(11) $1323 \cong 3132$ \
(12) $2313 \cong 3231$ 

\item[(13)] $2323 \cong 3131$ \quad
(14) $1331 \cong 2332$ \quad
(15) $1332$ \quad
(16) $2331$ 

\item[(17)] $3113 \cong 3223$ \quad
(18) $3123$ \quad
(19) $3213$ \quad
(20) $1113 \cong 3222$

\item[(21)] $1123 \cong 3122$ \
(22) $1213 \cong 3212$ \
(23) $2113 \cong 3221$ \
(24) $1223 \cong 3112$ 

\item[(25)]
$2123 \cong 3121$ \
(26) $2213 \cong 3211$ \
(27) $2223 \cong 3111$ 
(28) $1131 \cong 2322$ 

\item[(29)] $1132 \cong 1322$ \
(30) $1231 \cong 2312$ \
(31) $2131 \cong 2321$ \
(32) $1232 \cong 1312$ 

\item[(33)] $2132 \cong 1321$ \
(34) $2231 \cong 2311$ \
(35) $2232 \cong 1311$ \
(36) $1111 \cong 2222$

\item[(37)] 
$1112 \cong 1222$ \
(38) $1121 \cong 2122$ \
(39) $1211 \cong 2212$ \
(40) $2111 \cong 2221$ 

\item[(41)] $1122$ \
(42) $1212$ \
(43) $1221 \cong 2112$ \
(44) $2121$ \
(45) $2211$

\end{enumerate}
\end{prop}

\begin{proof}
This is a straightforward but tedious 
application of Proposition~\ref{prop:tiso}. 
\end{proof}

A partial map $\alpha : X \to X$ will below be 
encoded by a binary word $ab$
where $a,b \in \overline{X}$ meaning that $\alpha(1) = a$ and 
$\alpha(2)=b$ where we put $a=3$ or $b=3$ when $\alpha(1)$ respectively 
$\alpha(2)$ is not defined. So, for instance, the word $23$ means the partial
map $\alpha$ with $\alpha(1)=2$ and $\alpha(2)$ is undefined.
Note that with this notation $\Pfun(X,X) = \{ 3     3,
     1     3,
     2     3,
     3     1,
     3     2,
     1     1,
     1     2,
     2     1,
     2     2 \}$.
In the next result, we determine the  
(weak) partial endomorphisms and the (weak) Hom-associative
structures for the first representative in 
each of the 45 isomorphism classes in Proposition 
\ref{prop:isomorphismclasses}. This is achieved by an 
elementary case by case analysis using simple MATLAB-programs,
the codes of which can be obtained from the author upon request.

\begin{thm}\label{thm:main}
(a) The set of weak partial endomorphisms for each of the first 
representatives in the 45 isomorphism classes in 
Proposition \ref{prop:isomorphismclasses} is:
\begin{enumerate}[widest=(4)]

\item[(1)] $\Pfun(X,X)$ \
(2) $\Pfun(X,X)$ \
(3) $\{ 3  3,
     1     3,
     2     3,
     3     1,
     3     2,
     1     2,
     2     1,
     2     2 \}$ 

\item[(4)] $\Pfun(X,X)$ \
(5) $\Pfun(X,X)$ \
(6) $\Pfun(X,X)$ \
(7) $\Pfun(X,X)$ 

\item[(8)] $\{ 3  3,
     1     3,
     2     3,
     3     1,
     3     2,
     1     2,
     2     1,
     2     2  \}$ \
(9) $\{ 3  3,
     1     3,
     2     3,
     3     1,
     3     2,
     1     2,
     2     1,
     2     2  \}$ \
     
\item[(10)] $\Pfun(X,X)$ \
(11) $\Pfun(X,X)$ 
(12) $\{ 3  3,
     1     3,
     2     3,
     3     1,
     3     2,
     1     2,
     2     1,
     2     2  \}$ 

\item[(13)] $\{ 3  3,
     1     3,
     2     3,
     3     1,
     3     2,
     1     2,
     2     1,
     2     2  \}$ \
(14) $\{ 3  3,
     1     3,
     3     1,
     3     2,
     1     1,
     1     2 \}$ 

\item[(15)] $\Pfun(X,X)$ \
(16) $\{ 3  3,
     1     3,
     2     3,
     3     1,
     3     2,
     1     2,
     2     1 \}$

\item[(17)] $\{ 3  3,
     1     3,
     2     3,
     3     1,
     3     2,
     1     1,
     1     2,
     2     2 \}$ \
(18) $\Pfun(X,X)$ \
(19) $\Pfun(X,X)$ \
     
\item[(20)] $\{ 3  3,
     1     3,
     2     3,
     3     1,
     3     2,
     1     1,
     1     2,
     2     2 \}$ \
(21) $\Pfun(X,X)$ \
(22) $\Pfun(X,X)$ 

\item[(23)] $\{ 3  3,
     1     3,
     2     3,
     3     1,
     3     2,
     1     2,
     2     2 \}$ \
(24) $\{ 3  3,
     1     3,
     2     3,
     3     1,
     3     2,
     1     1,
     1     2,
     2     2 \}$ \
     
\item[(25)] $\{ 3  3,
     1     3,
     2     3,
     3     1,
     3     2,
     1     2,
     2     1,
     2     2 \}$ \
(26) $\{ 3  3,
     1     3,
     2     3,
     3     1,
     3     2,
     1     2,
     2     1,
     2     2 \}$
     
\item[(27)] $\{ 3  3,
     1     3,
     2     3,
     3     1,
     3     2,
     1     2,
     2     1,
     2     2 \}$ \
(28) $\{ 3  3,
     1     3,
     3     1,
     3     2,
     1     1,
     1     2 \}$ 
     
\item[(29)] $\Pfun(X,X)$ \
(30) $\{ 3  3,
     1     3,
     3     1,
     3     2,
     1     1,
     1     2 \}$ \
(31) $\{ 3  3,
     1     3,
     2     3,
     3     1,
     3     2,
     1     2,
     2     1 \}$
     
\item[(32)] $\Pfun(X,X)$ \
(33) $\{ 3  3,
     1     3,
     2     3,
     3     2,
     1     2,
     2     2 \}$ \
(34) $\{ 3  3,
     1     3,
     2     3,
     3     1,
     3     2,
     1     2,
     2     1 \}$
     
\item[(35)] $\{ 3  3,
     1     3,
     2     3,
     3     2,
     1     2,
     2     2 \}$ \
(36) $\{ 3  3,
     1     3,
     3     1,
     3     2,
     1     1,
     1     2 \}$ 
     
\item[(37)] $\{ 3  3,
     1     3,
     2     3,
     3     1,
     3     2,
     1     1,
     1     2,
     2     2 \}$ \
(38) $\{ 3  3,
     1     3,
     3     1,
     3     2,
     1     1,
     1     2 \}$ 
     
\item[(39)] $\{ 3  3,
     1     3,
     3     1,
     3     2,
     1     1,
     1     2 \}$ \
(40) $\{ 3  3,
     1     3,
     2     3,
     3     1,
     3     2,
     1     2 \}$
(41) $\Pfun(X,X)$
     
\item[(42)] $\Pfun(X,X)$ \
(43) $\{ 3  3,
     1     3,
     3     1,
     3     2,
     1     1,
     1     2 \}$ 
     
\item[(44)] $\{ 3  3,
     1     3,
     2     3,
     3     1,
     3     2,
     1     2,
     2     1 \}$ \
(45) $\{ 3  3,
     1     3,
     2     3,
     3     1,
     3     2,
     1     2,
     2     1 \}$
       
\end{enumerate}
(b) The set of partial endomorphisms
for each of the first 
representatives in the 45 isomorphism classes in 
Proposition \ref{prop:isomorphismclasses} is:
\begin{enumerate}[widest=(4)]

\item[(1)] $\Pfun(X,X)$ \
(2) $\{ 3  3,
     1     3,
     3     2,
     1     2 \}$ \
(3) $\{ 3  3,
     2     3,
     1     2 \}$ \
(4) $\{ 3  3,
     3     1,
     3     2,
     1     2 \}$
\item[(5)] $\{ 3  3,
     1     3,
     2     3,
     1     2 \}$ \
(6) $\{ 3  3,
     3     2,
     1     2 \}$ \
(7) $\{ 3  3,
     1     3,
     1     2 \}$ \
(8) $\{ 3  3,
     1     2 \}$ 
     
\item[(9)] $\{ 3  3,
     2     3,
     1     2 \}$ \
(10) $\{ 3  3,
     3     2,
     1     2 \}$ \
(11) $\{ 3  3,
     1     3,
     1     2 \}$ \
(12) $\{ 3  3,
     1     2 \}$

\item[(13)] $\{ 3  3,
     2     3,
     1     2  \}$ \ 
(14) $\{ 3  3,
     1     2 \}$ \
(15) $\{ 3  3,
     1     3,
     2     3,
     3     1,
     3     2,
     1     2,
     2     1 \}$
     
\item[(16)] $\{ 3  3,
     1     2,
     2     1 \}$ \
(17) $\{ 3  3,
     3     1,
     3     2,
     1     2 \}$ \
(18) $\{ 3  3,
     1     2,
     2     1 \}$ \
(19) $\{ 3  3,
     1     2,
     2     1 \}$
     
\item[(20)] $\{ 3  3,
     3     2,
     1     2 \}$ \
(21) $\{ 3  3,
     1     2 \}$ \
(22) $\{ 3  3,
     1     2 \}$ \
(23) $\{ 3  3,
     1     2 \}$ 
     
\item[(24)] $\{ 3  3,
     1     3,
     1     2 \}$
(25) $\{ 3  3,
     1     2 \}$ \
(26) $\{ 3  3,
     1     2 \}$ \
(27) $\{ 3  3,
     2     3,
     1     2 \}$

\item[(28)] $\{ 3  3,
     1     2 \}$ \
(29) $\{ 3  3,
     3     1,
     3     2,
     1     2 \}$ \
(30) $\{ 3  3,
     1     2 \}$ \
(31) $\{ 3  3,
     1     2 \}$ 
     
\item[(32)] $\{ 3  3,
     1     3,
     2     3,
     1     2 \}$ \
(33) $\{ 3  3,
     1     2 \}$ \
(34) $\{ 3  3,
     1     2 \}$ \
(35) $\{ 3  3,
     1     2 \}$
     
\item[(36)] $\{ 3  3,
     1     1,
     1     2 \}$ \
(37) $\{ 3  3,
     3     1,
     3     2,
     1     1,
     1     2,
     2     2 \}$ \
(38) $\{ 3  3,
     1     1,
     1     2 \}$
     
\item[(39)] $\{ 3  3,
     1     1,
     1     2 \}$ \
(40) $\{ 3  3,
     1     2 \}$ \
(41) $\{ 3  3,
     1     1,
     1     2,
     2     1,
     2     2 \}$
     
\item[(42)] $\{ 3  3,
     1     1,
     1     2,
     2     1,
     2     2 \}$ \
(43) $\{ 3  3,
     1     1,
     1     2 \}$
(44) $\{ 3  3,
     1     2,
     2     1 \}$ \
(45) $\{ 3  3,
     1     2,
     2     1 \}$
       
\end{enumerate}
(c) The set of partial Hom-associative structures for each 
of the first representatives in the 45 isomorphism classes in 
Proposition \ref{prop:isomorphismclasses} is:
\begin{enumerate}[widest=(4)]

\item[(1)] $\Pfun(X,X)$ \
(2) $\Pfun(X,X)$ \
(3) $\Pfun(X,X)$ \
(4) $\Pfun(X,X)$ 

\item[(5)] $\Pfun(X,X)$ \
(6) $\Pfun(X,X)$ \
(7) $\{ 3     3,
     1     3,
     2     3,
     3     1,
     3     2,
     1     2,
     2     1,
     2     2 \}$
     
\item[(8)] $\Pfun(X,X)$ \quad
(9)  $\Pfun(X,X)$ \quad
(10)  $\Pfun(X,X)$ \quad
     
\item[(11)] $\{ 3     3,
     1     3,
     2     3,
     3     1,
     3     2,
     1     2,
     2     1,
     2     2 \}$ \quad
(12) $\Pfun(X,X)$ \quad
(13) $\Pfun(X,X)$ \quad

\item[(14)] $\Pfun(X,X)$ \
(15) $\Pfun(X,X)$ \
(16) $\Pfun(X,X)$ \
(17) $\Pfun(X,X)$ 

\item[(18)] $\Pfun(X,X)$ \
(19) $\Pfun(X,X)$ \
(20) $\Pfun(X,X)$ \
(21) $\{ 3     3,
     1     3,
     3     2,
     1     2 \}$
     
\item[(22)] $\{ 3     3,
     1     3,
     3     2,
     1     2 \}$ \quad
(23) $\{ 3     3,
     1     3,
     2     3,
     3     1,
     3     2,
     1     2,
     2     1,
     2     2 \}$ 
     
\item[(24)] $\{ 3     3,
     1     3,
     2     3,
     3     1,
     3     2,
     1     2,
     2     1,
     2     2 \}$ \quad
(25) $\{ 3     3,
     2     3,
     3     1,
     3     2,
     2     1,
     2     2 \}$
     
\item[(26)] $\{ 3     3,
     2     3,
     3     1,
     3     2,
     2     1,
     2     2 \}$ \quad
(27) $\Pfun(X,X)$ \quad
(28) $\Pfun(X,X)$
     
\item[(29)] $\{ 3     3,
     1     3,
     2     3,
     3     1,
     3     2,
     1     1,
     1     2 \}$ \quad
(30) $\{ 3     3,
     1     3,
     2     3,
     3     1,
     3     2,
     1     2,
     2     1 \}$
     
\item[(31)] $\{ 3     3,
     1     3,
     2     3,
     3     1,
     3     2,
     1     1,
     2     1,
     2     2 \}$ \quad
(32) $\{ 3     3,
     1     3,
     2     3,
     3     1,
     3     2,
     1     2,
     2     2 \}$
     
\item[(33)] $\{ 3     3,
     1     3,
     2     3,
     3     1,
     3     2,
     1     2,
     2     1 \}$ \quad
(34) $\{ 3     3,
     1     3,
     2     3,
     3     1,
     3     2,
     1     1,
     2     1,
     2     2 \}$
     
\item[(35)] $\Pfun(X,X)$ \quad
(36) $\Pfun(X,X)$ \quad
(37) $\{ 3     3,
     1     3,
     2     3,
     3     1,
     3     2,
     1     1,
     1     2 \}$
     
\item[(38)] $\{ 3     3,
     1     3 \}$ \quad
(39) $\{ 3     3,
     1     3 \}$ \quad
(40) $\{ 3     3,
     1     3,
     2     3,
     3     1,
     3     2,
     2     2 \}$ 
     
\item[(41)] $\{ 3     3,
     1     3,
     3     2,
     1     2 \}$ \
(42) $\{ 3     3,
     1     3,
     3     2,
     1     2 \}$ \
(43) $\{ 3     3,
     1     3,
     2     3,
     3     1,
     3     2,
     1     2,
     2     1 \}$
\item[(44)] $\{ 3     3,
     2     3,
     3     1,
     2     1 \}$ \
(45) $\{ 3     3,
     2     3,
     3     1,
     2     1 \}$
       
\end{enumerate}
In the 37 cases (1)-(24), (27)-(30), (32), (33), (35)-(37) and
(41)-(43) $(X,\nabla)$ is partially associative.

\vspace{1mm}

\noindent 
(d) The set of Hom-associative structures for each of the first 
representatives in each the 45 isomorphism classes in 
Proposition \ref{prop:isomorphismclasses} is:
\begin{enumerate}[widest=(4)]

\item[(1)] $\Pfun(X,X)$ \quad 
(2) $\{ 3              3,       
       1              3,       
       2              3,      
       3              2,           
       1              2,        
       2              2 \}$ \quad
(3) $\Pfun(X,X)$ 

\item[(4)] $\{ 3              3,       
       1              3,         
       3              1,     
       1              1 \}$ \
(5) $\{ 3              3,       
       2              3,      
       3              2,      
       2              2 \}$ \
(6) $\{ 3     3 \}$ \
(7) $\{ 33,      
       1    2 \}$ \
(8) $\{ 3     3 \}$
     
\item[(9)] $\{ 3              3,       
       2              3,       
       3              2,       
       2              2 \}$ \quad
(10) $\{ 3     3 \}$ \quad
(11) $\{ 3              3,     
       1              2 \}$ \quad
(12) $\{ 3     3 \}$ 

\item[(13)] $\{ 3              3,       
       2              3,       
       3              2,       
       2              2 \}$ \quad
(14) $\{ 3              3,
       2              3,       
       3              2,       
       2              2 \}$ \quad
(15) $\{ 3              3,       
       1              3,         
       3              2,      
       1              2 \}$
     
\item[(16)] $\{ 3              3,       
       2              3,       
       3              1,       
       2              1 \}$ \
(17) $\{ 3              3,       
       1              3,             
       3              1,             
       1              1 \}$ \      
(18) $\{ 3     3 \}$ \
(19) $\{ 3     3 \}$ \
(20) $\{ 3     3 \}$
     
\item[(21)] $\{ 3     3 \}$ \
(22) $\{ 3     3 \}$ \ 
(23) $\{ 3     3 \}$ \  
(24) $\{  3              3,         
       2              3,   
       1              2 \}$
(25)  $\{ 3     3 \}$ \
(26)  $\{ 3     3 \}$    
     
\item[(27)]  $\{ 33,       
       2              3,       
       3              2,       
       2              2 \}$ \
(28)  $\{ 3     3 \}$ \
(29)  $\{ 3     3 \}$ \
(30)  $\{ 3     3 \}$ \
(31)  $\{ 3     3 \}$ \
(32)  $\{ 3     3 \}$ 
     
\item[(33)]  $\{ 3     3 \}$ \ 
(34)  $\{ 3     3 \}$ \
(35)  $\{ 3     3 \}$ \
(36)  $\{ 33,       
       1              1,       
       1              2,       
       2              1,       
       2              2 \}$ \
(37) $\{ 33,         
       1              1,       
       1              2 \}$     

\item[(38)] $\{ 3     3 \}$ \quad 
(39) $\{ 3     3 \}$ \quad
(40) $\{ 3 3,       
       2 2 \}$ \quad   
(41) $\{ 33, 12 \}$ \quad
(42) $\{ 33,       
       12 \}$ 
       
\item[(43)] $\{ 33,       
       12,       
       21 \}$ \quad
(44) $\{ 33,    
         21 \}$ \quad
(45) $\{ 33,       
         21 \}$     
       
\end{enumerate}
In the 13 cases (1)-(3), (7), (11), (15),
(24), (36), (37) and (41)-(43)   
$(X,\nabla)$ is associative.
\end{thm}

\section{Hom-associative magma algebras}\label{sec:magmaalgebras}

\subsection{Hom-algebras}

For the rest of this paper, $K$ denotes a field
and $A$ denotes a $K$-vector space.
Suppose that $\mu$ is a $K$-bilinear map 
$A \times A \to A$ and that $\tau$ is a $K$-linear map $A \to A$.
Recall from \cite{makhlouf2010} that the triple 
$(A,\mu,\tau)$ is then called a \emph{Hom-algebra}.
From \cite{makhlouf2010,makhlouf2020}, we extract the following:

\begin{defi}\label{def:homproperties}
A Hom-algebra $(A,\mu,\tau)$ is said to be: 
\begin{itemize}[topsep=1mm,itemsep=1mm]

\item \emph{multiplicative}
if $\tau \circ \mu = \mu \circ (\tau \times \tau)$; 

\item \emph{Hom-associative} if 
$\mu \circ (\tau \times \mu) = \mu \circ (\mu \times \tau)$;

\item \emph{associative} if $(A,\mu,\I)$ is 
Hom-associative. 

\end{itemize}
\end{defi}

From now on, we fix a $K$-vector space basis 
$B = \{ e_i \}_{i \in I}$ for $A$. The 
maps $\mu$ and $\tau$ are uniquely determined
by their \emph{structure constants} $C_{ij}^k,t_{ij} \in K$ 
which are defined by 
$\mu(e_i,e_j) = \sum_{k \in I} C_{ij}^k e_k$ and
$\tau(e_i) = \sum_{j \in I} t_{ij} e_j$
for $i,j \in I$. Note that given $i,j \in I$ ($i \in I$),
then $C_{ij}^k = 0$ ($t_{ij} = 0$) for all but finitely
many $k \in I$ ($j \in I$) making the sums above well defined.
The properties in Definition \ref{def:homproperties} can now
be formulated using structure constants. Indeed,
from the discussion in \cite[Section 4.4]{makhlouf2020}
we extract the following:

\begin{prop}\label{prop:structure}
A Hom-algebra $(A,\mu,\tau)$ is:
\begin{itemize}[widest=(a),topsep=1mm,itemsep=1mm]

\item[{\rm (a)}] multiplicative if and only if
for all $i,j,s \in I$
\[
\sum_{p \in I} t_{sp} C_{ij}^p  = 
\sum_{p,q \in I} t_{pi} t_{qj} C_{pq}^s;
\]

\item[{\rm (b)}] Hom-associative if and only if
for all $i,j,k,s \in I$
\[
\sum_{l,m \in I} t_{il} C_{jk}^m C_{lm}^s 
= \sum_{l,m \in I} t_{mk} C_{ij}^l C_{lm}^s;
\]

\item[{\rm (c)}] associative if and only if for all $i,j,k,s \in I$
\[
\sum_{m \in I} C_{jk}^m C_{im}^s  =
\sum_{l \in I} C_{ij}^l C_{lk}^s.
\]

\end{itemize}
\end{prop}

We now introduce the following weakening of 
Definition \ref{def:homproperties}:

\begin{defi}\label{def:partilconcepts}
We say that a Hom-algebra $(A,\mu,\tau)$ is: 
\begin{itemize}[topsep=1mm,itemsep=1mm]

\item \emph{partially $B$-multiplicative} if
$\tau(\mu(a,b)) = \mu(\tau(a),\tau(b))$ holds for all $a,b \in B$ 
with $\tau(\mu(a,b)) \neq 0 \neq \mu(\tau(a),\tau(b))$;

\item \emph{partially $B$-Hom-associative} if
$\mu(\tau(a),\mu(b,c)) = \mu(\mu(a,b),\tau(c))$ holds
for all $a,b,c \in B$ with 
$\mu(\tau(a),\mu(b,c)) \neq 0 \neq \mu(\mu(a,b),\tau(c))$;

\item \emph{partially $B$-associative} if
$(A,\mu,\I)$ is partially $B$-Hom-associative.

\end{itemize}
\end{defi}

In analogy with Proposition \ref{prop:structure},
it is a straightforward exercise 
to see that the properties in Definition
\ref{def:partilconcepts} can be formulated using
structure constants:

\begin{prop}
A Hom-algebra $(A,\mu,\tau)$ is:
\begin{itemize}[widest=(a),topsep=1mm,itemsep=1mm]

\item[{\rm (a)}] partially $B$-multiplicative 
if and only if for all $i,j \in I$ with both
\[
\Biggl( \sum_{p \in I} t_{sp} C_{ij}^p \Biggr)_{s \in I} 
\quad \mbox{and} \quad 
\Biggl( \sum_{p,q \in I} t_{pi} t_{qj} C_{pq}^s \Biggr)_{s \in I}
\quad \mbox{nonzero},
\]
then $\left( \sum_{p \in I} t_{sp} C_{ij}^p \right)_{s \in I} = 
\left( \sum_{p,q \in I} t_{pi} t_{qj} C_{pq}^s \right)_{s \in I}$;

\item[{\rm (b)}] partially $B$-Hom-associative 
if and only if for all $i,j,k \in I$ with both
\[
\Biggl( \sum_{l,m \in I} t_{il} C_{jk}^m C_{lm}^s 
\Biggr)_{s \in I}
\quad \mbox{and} \quad
\Biggl( \sum_{l,m \in I} t_{mk} C_{ij}^l C_{lm}^s 
\Biggr)_{s \in I}
\quad \mbox{nonzero}, 
\]
then $\left( \sum_{l,m \in I} t_{il} C_{jk}^m C_{lm}^s 
\right)_{s \in I} 
= \left(\sum_{l,m \in I} t_{mk} C_{ij}^l C_{lm}^s
\right)_{s \in I}$;

\item[{\rm (c)}] partially $B$-associative 
if and only if for all $i,j,k \in I$ with both
\[
\Biggl( \sum_{m \in I} C_{jk}^m C_{im}^s \Biggr)_{s \in I}
\quad \mbox{and} \quad
\Biggl( \sum_{l \in I} C_{ij}^l C_{lk}^s \Biggr)_{s \in I}
\quad \mbox{nonzero}, 
\]
then $\left( \sum_{m \in I} C_{jk}^m C_{im}^s \right)_{s \in I} =
\left( \sum_{l \in I} C_{ij}^l C_{lk}^s \right)_{s \in I}$.

\end{itemize}
\end{prop}

\subsection{Magma algebras}

Suppose that $(X,\nabla)$ is a partial magma and let
$\alpha \in \Pfun(X,X)$. We now show  
that $\nabla$ and $\alpha$ induce, in a natural way, 
a Hom-algebra structure on the so called
magma algebra $K[X]$ of $X$ over $K$ 
(see Theorem \ref{thm:secondmain} below). 
Recall that the elements of $K[X]$ are formal sums
$\sum_{x \in X} k_x x$, for some $k_x \in K$, satisfying
$k_x = 0$ for all but finitely many $x \in X$.
Take $k \in K$. Suppose that $a := \sum_{x \in X} l_x x \in K[X]$ and 
$b := \sum_{x \in X} m_x x \in K[X]$. If we put
\[
k a = \sum_{x \in X}  (k l_x) x \quad \quad \mbox{and} \quad \quad
a+b = \sum_{x \in X} (l_x+m_x)x,
\]
then, with these operations, $K[X]$ is a $K$-vector space
having the elements of $B := X$ as a basis.
Let $\tau_\alpha : K[X] \to K[X]$ be defined in the following way.
Take $x \in X$. Put
$\tau_\alpha(x) = \alpha(x)$, if $\alpha(x)$ is defined, and
$\tau_\alpha(x) = 0$, otherwise. 
Using structure constants this means that
for all $x,y \in X$:
\[
t_{xy} = 
\left\{
\begin{array}{cl}

1 & \mbox{if $\alpha(x)$ is defined and $y = \alpha(x)$;} 
\\[1mm]
0 & \mbox{if $\alpha(x)$ is not defined, 
or $\alpha(x)$ is defined but $y \neq \alpha(x)$.} 

\end{array}
\right.
\]
Then we $K$-linearly extend $\tau_\alpha$ to $K[X]$.
Let $\mu_\nabla : K[X] \times K[X] \to K[X]$ be defined
in the following way. Take $x,y \in X$. Put
$\mu_\nabla(x,y) = \nabla(x,y)$, if $\nabla(x,y)$ is defined,
and $\mu_\nabla(x,y) = 0$, otherwise.
In the language of structure constants, 
this amounts to saying that for all $x,y,z \in X$:
\[
C_{xy}^z = 
\left\{
\begin{array}{cl}
1 & \mbox{if $\nabla(x,y)$ is defined and $z = \nabla(x,y)$;} 
\\[1mm]
0 & \mbox{if $\nabla(x,y)$ is not defined,
or $\nabla(x,y)$ is defined but $z \neq \nabla(x,y)$.} 
\end{array}
\right.
\]
Then we $K$-bilinearly extend $\mu_\nabla$ to $K[X] \times K[X]$.
With the above notations: 

\begin{thm}\label{thm:secondmain}
Let $(X,\nabla)$ be a partial magma and let
$\alpha \in \Pfun(X,X)$. 
\begin{itemize}[widest=(a),topsep=1mm,itemsep=1mm]

\item[{\rm (a)}] $(K[X],\mu_\nabla,\tau_\alpha)$ is a Hom-algebra;

\item[{\rm (b)}] $(K[X],\mu_\nabla,\tau_\alpha)$ is (partially 
$X$-)multiplicative if and only if 
$\alpha$ is a (weak) partial endomorphism of partial magmas;

\item[{\rm (c)}] $(K[X],\mu_\nabla,\tau_\alpha)$ is (partially
$X$-)Hom-associative if and only if 
$(X,\nabla,\alpha)$ is (partially) Hom-associative;

\item[{\rm (d)}] $(K[X],\mu_\nabla)$ is partially
$X$-associative if and only if 
$(X,\nabla)$ is partially associative;

\item[{\rm (e)}] $(K[X],\mu_\nabla)$ is associative 
if and only if $(X,\nabla)$ is associative.

\end{itemize}
\end{thm}

\begin{proof}
(a) This is clear. The statement in (b) follows 
from $K$-linearity and the fact that the equalities
\[ 
\begin{array}{rcl}
(\tau_\alpha \circ \mu_\nabla)(x,y) &=& 
(\alpha \circ \nabla)(x,y) \\[1mm]
(\mu_\nabla \circ (\tau_\alpha \times \tau_\alpha))(x,y) 
&=& (\nabla \circ (\alpha \times \alpha))(x,y) 
\end{array}
\]
hold for all $x,y \in X$ for which the left hand sides above 
are nonzero. The statement in (c) follows from $K$-bilinearity and the 
fact that the equalities
\[ 
\begin{array}{rcl}
( \mu_\nabla \circ (\tau_\alpha \times \mu_\nabla) )(x,y,z)
&=& (\nabla \circ (\alpha \times \nabla))(x,y,z) \\[1mm]
(\mu_\nabla \circ (\mu_\nabla \times \tau_\alpha))(x,y,z) 
&=& (\nabla \circ (\nabla \times \alpha))(x,y,z)
\end{array}
\]
hold for all $x,y,z \in X$ for which the left hand sides above
are nonzero. The statements in (d) and (e) follow from (c).
\end{proof}

\begin{exa}\label{ex:someinstances}
We now exemplify Theorems \ref{thm:main} and \ref{thm:secondmain} 
for some instances of Hom-algebras $H := (K[X],\mu_\nabla,\tau_\alpha)$ 
when $(X,\nabla)$ is a magma of order two. 
\begin{itemize}[widest=(a),itemsep=1mm,topsep=1mm]
\item[(a)] Let $\nabla = 2232$. This is magma (35) in Proposition 
\ref{prop:isomorphismclasses}. Then:
\begin{itemize}[topsep=0mm,itemsep=0.5mm]

\item $H$ is partially $X$-multiplicative $\Leftrightarrow$
$\alpha \in \{ 33,13,23,32,12,22 \}$; 
     
\item $H$ is multiplicative $\Leftrightarrow$
$\alpha \in \{ 33,12 \}$;

\item $H$ is partially $X$-Hom-associative $\Leftrightarrow$
$\alpha \in {\rm Pfun}(X,X)$;

\item $H$ is Hom-associative $\Leftrightarrow$
$\alpha = 33$;

\item $H$ is multiplicative and Hom-associative $\Leftrightarrow$
$\alpha = 33$;

\item $H$ is partially $X$-associative but not associative.

\end{itemize}
\item[(b)] Let $\nabla = 2111$. This is magma (40) in Proposition 
\ref{prop:isomorphismclasses}. Then:
\begin{itemize}[topsep=0mm,itemsep=0.5mm]

\item $H$ is partially $X$-multiplicative $\Leftrightarrow$
$\alpha \in \{ 33,13,23,31,32,12 \}$; 
     
\item $H$ is multiplicative $\Leftrightarrow$
$\alpha \in \{ 33,12 \}$;

\item $H$ is partially $X$-Hom-associative $\Leftrightarrow$
$\alpha \in \{ 33,13,23,31,32,22 \}$;

\item $H$ is Hom-associative $\Leftrightarrow$
$\alpha \in \{ 33,22 \}$;

\item $H$ is multiplicative and Hom-associative $\Leftrightarrow$
$\alpha = 33$;

\item $H$ is not partially $X$-associative and hence not 
associative.

\end{itemize}
\item[(c)] Let $\nabla = 1221$. This is magma (43) in Proposition 
\ref{prop:isomorphismclasses}. Note that $(X,\nabla)$ is a group 
so that $H$ is the group ring $K[X]$. Then:
\begin{itemize}[topsep=0mm,itemsep=0.5mm]

\item $H$ is partially $X$-multiplicative $\Leftrightarrow$
$\alpha \in \{ 3  3,
     1     3,
     3     1,
     3     2,
     1     1,
     1     2 \}$; 
     
\item $H$ is multiplicative $\Leftrightarrow$
$\alpha \in  \{ 33,11,12 \}$;

\item $H$ is partially $X$-Hom-associative $\Leftrightarrow$
$\alpha \in  \{ 3     3,
     1     3,
     2     3,
     3     1,
     3     2,
     1     2,
     2     1 \}$;

\item $H$ is Hom-associative $\Leftrightarrow$
$\alpha \in \{ 33, 12 , 21 \}$;

\item $H$ is multiplicative and Hom-associative $\Leftrightarrow$
$\alpha \in \{ 33 , 12 \}$; 

\item $H$ is associative and hence partially $X$-associative.

\end{itemize}
\item[(d)] Let $\nabla = 2121$. This is magma (44) in Proposition 
\ref{prop:isomorphismclasses}. Then:
\begin{itemize}[topsep=0mm,itemsep=0.5mm]

\item $H$ is partially $X$-multiplicative $\Leftrightarrow$
$\alpha \in \{ 3  3,
     1     3,
     2     3,
     3     1,
     3     2,
     1     2,
     2     1 \}$; 
     
\item $H$ is multiplicative $\Leftrightarrow$
$\alpha \in \{ 33,12,21 \}$;

\item $H$ is partially $X$-Hom-associative $\Leftrightarrow$
$\alpha \in \{ 33,23,31,21 \}$;

\item $H$ is Hom-associative $\Leftrightarrow$
$\alpha \in \{ 33,21 \}$;

\item $H$ is multiplicative and Hom-associative $\Leftrightarrow$
$\alpha \in \{ 33,21 \}$;

\item $H$ is not partially $X$-associative and hence not 
associative.

\end{itemize}
\item[(e)] Let $\nabla = 2211$. This is magma (45) in Proposition 
\ref{prop:isomorphismclasses}. Then:
\begin{itemize}[topsep=0mm,itemsep=0.5mm]

\item $H$ is partially $X$-multiplicative $\Leftrightarrow$
$\alpha \in \{ 33,13,23,31,32,12 \}$; 
     
\item $H$ is multiplicative $\Leftrightarrow$
$\alpha \in \{ 33,12 \}$;

\item $H$ is partially $X$-Hom-associative $\Leftrightarrow$
$\alpha \in \{ 33,13,23,31,32,22 \}$;

\item $H$ is Hom-associative $\Leftrightarrow$
$\alpha \in \{ 33,22 \}$;

\item $H$ is multiplicative and Hom-associative $\Leftrightarrow$
$\alpha = 33$;

\item $H$ is not partially $X$-associative and hence not 
associative.

\end{itemize}
\end{itemize}
\end{exa}



\end{document}